\documentclass{paper}
\usepackage[english]{babel}
\usepackage{enumerate,amsmath,amsthm,amsfonts,pifont,slashed,amssymb,ifsym,mathrsfs,graphicx,graphics,yfonts,calligra,
latexsym,txfonts
}
\usepackage{authblk}
\usepackage{hyphenat}
\usepackage[utf8]{inputenc}

\usepackage{array}
\usepackage{amssymb}
\usepackage{enumerate}
\usepackage{fancyhdr}
\usepackage{layout}
\usepackage{multicol}
\usepackage{pstricks}
\usepackage{tikz-cd}

\usepackage[colorlinks]{hyperref}
 \hypersetup{linkcolor=blue}
\usepackage[all]{xy}
\usepackage[sort]{cite}

\newtheorem{theorem}{Theorem}[section]
\newtheorem{proposition}[theorem]{Proposition}
\newtheorem{lemma}[theorem]{Lemma}
\newtheorem{corollary}[theorem]{Corollary}

\theoremstyle{definition}
\newtheorem{definition}[theorem]{Definition}
\newtheorem{exm}[theorem]{Example}

\theoremstyle{remark}
\newtheorem{remark}[theorem]{Remark}

\newcommand{\lcc}{\operatorname{lcc}}

\newcommand{\op}{\operatorname{op}}
\newcommand{\id}{\operatorname{id}}

\newcommand{\Max}{\operatorname{Max}}
\newcommand{\Rad}{\operatorname{Rad}}

\newcommand{\clopI}{\operatorname{Clop}_{\uparrow}}
\newcommand{\clop}{\operatorname{Clop}}
\newcommand{\LL}{\text{\L}}

\newcommand{\MVs}{\mathcal{MV}^{\operatorname{ss}}}
\newcommand{\MVlcc}{\mathcal{MV}^{\operatorname{lcc}}}
\newcommand{\TMV}{{}^{\operatorname{MV}}\!\mathcal{T}\!\!\operatorname{op}}

\newcommand{\PMV}{{}^{\operatorname{MV}}\!\ \mathcal{T}\!\!\operatorname{op}_{\leq}}
\newcommand{\Pries}{{}^{\operatorname{MV}}\!\mathcal{P}\!\!\operatorname{ries}}
\newcommand{\MVL}{{}^{\operatorname{MV}}\!\mathcal{L}}

\newcommand{\SMV}{{}^{\operatorname{MV}}\!\!\mathcal{S}\!\operatorname{tone}}

\renewcommand{\hom}{\operatorname{Hom}}

\newcommand{\e}{\varepsilon}
\renewcommand{\O}{\Omega}

\newcommand{\lto}{\longrightarrow}

\newcommand{\lmapsto}{\longmapsto}
\newcommand{\cou}{{^\leftarrow}}
 
\newcommand{\0}{\mathbf{0}}

\renewcommand{\1}{\mathbf{1}}

\newcommand{\n}{\mathbf}
\newcommand{\MV}{\mathcal{MV}}

\newcommand{\N}{\mathbb{N}}

\title{An extension of Priestley duality to fuzzy topologies and positive MV-algebras}
\author[1]{Marby Zuley Bolaños Ortiz}
\affil[1]{Departamento de Matem\'aticas, Universidad del Valle -- Cali, Colombia \newline marby.bolanos@correounivalle.edu.co}
\author[2]{Ciro Russo}
\affil[2]{Departamento de Matem\'atica, Universidade Federal da Bahia -- Salvador, BA, Brazil \newline ciro.russo@ufba.br}

\date{}
\begin{document}
\maketitle
\begin{abstract}
We extend Priestley Duality to suitable categories of fuzzy topological spaces and positive MV-algebras, which stand to MV-algebras as bounded distributive lattices stand to Boolean algebras. The duality we prove extends not only classical Priestley Duality between Priestley Spaces and bounded distributive lattices, but also the duality between limit cut complete MV-algebras and Stone MV-topological spaces \cite{ciro2016} which, on its turn, is an extension of classical Stone Duality.
\end{abstract}

\section{Introduction}
The relationship between topology and lattice theory began to be explored through the works of M. H. Stone \cite{stone1934,stone1936,stonedistr}. 
In \cite{doctor64}, Doctor reviewed Stone's results from a categorical viewpoint, thus showing the existence of a dual equivalence between the categories of Boolean algebras, with Boolean homomorphisms as morphisms, and Stone spaces, i. e., Hausdorff, compact, zero-dimensional spaces, with continuous maps as arrows. Later on, Priestley \cite{pri70} proved a topological representation of bounded distributive lattices where the topological characterization of the dual spaces was simpler than the one in \cite{stonedistr} and explicitly incorporates an order relation. This categorical equivalence is known as Priestley duality.

In \cite{ciro2016}, the second author defined a class of fuzzy topological spaces, called \emph{MV-topological spaces}, in order to extend Stone Duality to fuzzy topology and MV-algebras. MV-algebras are arguably most natural generalization of Boolean algebras in the realm of fuzzy logics. More precisely, in \cite{ciro2016} an adjunction between the category of MV-algebras and the one of MV-topological spaces, which are fuzzy spaces in the sense of Chang \cite{changfuzzy}, was proved. Such an adjunction induces a duality between the full subcategories of limit cut complete (lcc) MV-algebras and Stone MV-topological spaces which, on its turn, coincide with classical Stone Duality when restricted to the categories of Boolean algebras and Stone spaces which are full subcategories of, respectively, lcc MV-algebras and Stone MV-spaces.

Eventually, MV-topologies have been further investigated \cite{tesis,dlprarc}, and were also used as base spaces for a representation of MV-algebras as sheaves of $\ell$-groups on fuzzy topological spaces \cite{dlprsoco}.

In the present paper, we aim at further exploring the connection between MV-algebra related structures and fuzzy topological spaces. In fact,
 we prove a Priestley-type duality for MV-topological spaces. Starting from the aforementioned Stone-type duality for MV-algebras, we introduce \emph{Priestley MV-spaces} and \emph{MV-lattices} --- a natural generalization of distributive lattices admitting further operations. It turns out that the MV-lattice duals of Priestley MV-spaces, are a subclass of \emph{positive MV-algebras} \cite{abbadini2022}, which are negation-free subreducts of MV-algebras. 

The paper is organised as follows. In Section \ref{prelsec}, we recall basic notions and results about MV-algebras and MV-topologies, along with the main results of \cite{ciro2016}, in order to make the paper as self-contained as possible. In Section \ref{MV-lattice}, we introduce and study MV-lattices, and recall some facts about positive MV-algebras. In particular, we prove some properties which will allow  us to associate a partially ordered space to any given MV-lattice.

The mains results are presented in Section \ref{dualsec}. In Theorem \ref{adjuncionP} we establish an adjunction between partially ordered MV-spaces and MV-lattices. Then we show how such a pair of functors restricts to a Priestley-type duality between a subcategory of Priestley MV-spaces and \emph{$H$-complete lcc positive MV-algebras} (Theorem \ref{dualth}).

Last, in Section\ref{compsec}, we discuss the connection between our result and both classical Priestley Duality and the Stone-type duality presented in \cite{ciro2016}.

\section{Preliminaries}\label{prelsec}
This section is devoted to introducing the basic concepts necessary for the subsequent development. For more detail on MV-algebras, we refer the reader to \cite{cignoli2013,MVdinola}, and for fuzzy topology and MV-topology to \cite{ftbook,changfuzzy,low} and \cite{ciro2016,dlprarc,dlprsoco,tesis} respectively.

\begin{definition}\label{mvalg}
An \emph{MV-algebra} is a structure $(A,\oplus,^*,0)$ where $\oplus$ is a binary operation, $^*$ is a unary operation and $0$ is a constant such that the following axioms are satisfied for any $a,b\in A$:
\begin{description}
  \item[(MV1)] $(A,\oplus,0)$ is an Abelian monoid,
  \item[(MV2)] $(a^*)^*=a,$
  \item[(MV3)] $0^*\oplus a=0^*,$
  \item[(MV4)] $(a^*\oplus b)^*\oplus b=(b^*\oplus a)^*\oplus a.$
\end{description}
\end{definition}
We denote an MV-algebra $(A,\oplus,^*,0)$ simply by its universe \emph{A}, and define the constant $1$ and the derived operations $\odot$  and $\ominus$ as follows:
\begin{itemize}
  \item $1:=0^*$
  \item $a\odot b:= (a^*\oplus b^*)^*$
  \item $a\ominus b:= a\odot b^*$
\end{itemize}

Let $A$ be an MV-algebra, $B$ a subalgebra of $A$, and $S \subseteq B$. If $B$ is the smallest subalgebra of $A$ which includes $S$, then $S$ is called a \emph{generating set} for $B$; in general, the smallest subalgebra of $A$ containing $S$ is called the \emph{subalgebra generated by} $S$ and is denoted by $\langle S \rangle$.
 
Now let $S^{*}$ be the set $\{s^{*}: s\in S \}$. The following characterization of generated subalgebras hold.
 \begin{proposition}\label{termo}
     Let $A$ be an MV-algebra and $S$ a nonempty subset of $A$. Set $S_{0}=S$ and, for all $ i \geq 1$, $S_{i}=S_{i-1} \cup S_{i-1}^* \cup \{a \oplus b : a,b \in S_{i-1}\}$. Then $\langle S \rangle=\bigcup\limits_{i \geq 0}S_{i}$. \end{proposition}
 Consequently, if $A=\langle S \rangle$, every $x \in A$ can be expressed as $x=t(a_{1},a_{2}, \cdots, a_{n} )$ where $t$ is a term in the language $\{\oplus, ^{*}\}$ and $a_{i} \in S$.

An \emph{ideal} of an MV-algebra $A$ is a subset $I$, is downward closed submonoid of $(A, \oplus, 0)$. It is well-known that ideals and congruences of every MV-algebra are in bijective correspondence. In fact, for any congruence, the class of $0$ is an ideal and, for every ideal $I$, the relation defined by $a \sim_I b$ if and only if $d(a,b) := (a \ominus b) \oplus (b \ominus a) \in I$ is a congruence and is the unique one whose $0$-class is $I$.

The set of ideals of an MV-algebra $A$ is ordered by inclusion, and an ideal is \emph{maximal} if it is a maximal element of the set of proper ideals of $A$. The set of maximal ideals of $A$ is denoted by $\Max A$, and the intersection of all maximal ideals (which is obviously an ideal too) is called the \emph{radical} of $A$ and denoted by $\Rad A$.

Now, we recall the definition of simple and semisimple MV-algebra.
\begin{definition}\label{ss}
An MV-algebra $A$ is called:
\begin{enumerate}
\item[(i)] \emph{simple} if its only proper ideal is $\{0\}$, and
\item[(ii)]  \emph{semisimple} if $A$ is nontrivial and $\Rad A = \{0\}$. 
\end{enumerate}
\end{definition}
The class of semisimple MV-algebras, with MV-algebra homomorphisms, forms a full subcategory of the one -- $\MV$ -- of MV-algebras. We denote such a subcategory by $\MVs$. 


Due to its relevance to our results, we recall the following representation theorem for semisimple MV-algebras, along with a sketch of its proof.
\begin{theorem}\cite{belluce}\label{bell repr}
Every semisimple MV-algebra $A$ can be embedded in the MV-algebra of fuzzy subsets  $[0,1]^{\Max A}$ of $\Max A$.
\end{theorem}
\begin{proof} (Sketch)
For each maximal ideal $M$ of $A$, the quotient algebra $A/M$ is simple and, therefore, isomorphic to a subalgebra of $[0,1]$. So, there exists an embedding $\iota_{M}:A/M \to [0,1]$. Let $\iota: A \to [0,1]^{\Max A}$ be the map defined by $\iota(a) = \hat{a}$ for each $a \in A$, where
\begin{align*}
\hat{a}: \Max A& \to [0,1]\\
         M& \mapsto \iota_{M}(a/M).
\end{align*}
The function $\iota$ is easily seen to be a homomorphism. Moreover
\begin{align*}
    \ker\iota=& \bigcap _{M\in \Max A}\{a \in A : \hat{a}(M)=0\} \\
    =&\bigcap _{M\in \Max A}\{a \in A: \iota_{M}(a)=(a/M)=0\}\\
    =&\bigcap _{M\in \Max A}M = \Rad A.
\end{align*}
Since $A$ is a semisimple MV-algebra, $\ker\iota = \Rad A = \{0\}$, whence $\iota$ is an embedding.
\end{proof}

\begin{remark}\label{B[0,1]}
It is well-known also that, if $A$ is a semisimple MV-algebra, then the map
    \begin{align*}
    T: \operatorname{Hom}_{\MV}(A, [0,1]) & \rightarrow \Max A\\
    f & \mapsto f^{-1}[0],
\end{align*} 
is a bijection. As a consequence, for $f \in \operatorname{Hom}_{\MV}(A,[0,1])$ and for each $a \in A$, $f(a)= \iota_{f^{-1}[0]}\left(a / f^{-1}[0]\right)$, and $\iota_{f^{-1}[0]}: A/f^{-1}[0]  \rightarrow[0,1]$ is the canonical embedding mentioned in Theorem \ref{bell repr}.
\end{remark}

Let $A$ be a semisimple MV-algebra, $X$ a subset of $A$ and $\leq$ the canonical partial order of $A$
. As usual, we  denote by $lX$ the set of lower bounds of $X$ in $A$ and by $uX$ the  set of all upper bounds of $X$ in $A$. In this part, for $a \in A$ and $X \subseteq A$ we adopt the notation $\hat{a}:=\iota(a) \in [0,1]^{\Max A}$ and $\widehat{X}:=\iota(X) \subseteq [0,1]^{\Max A}$, according to the proof of Theorem \ref{bell repr}.
\begin{definition}\label{ducuco}
Let $A$ be a semisimple MV-algebra. A subset $X$ of $A$ is called a \emph{cut} if $X = luX$. A cut $X$ of $A$ is a \emph{limit cut} iff
\begin{equation}\label{distance}
 d(\widehat X,\widehat{uX}) = \bigwedge \{d(\widehat a, \widehat b) \mid b \in uX, a \in X\} = \bigwedge \{\widehat b \ominus \widehat a \mid b \in uX, a \in X\} = 0.
\end{equation}

We shall say that $A$ is \emph{limit cut complete} (\emph{lcc} for short) if, for any limit cut $X$ of $A$, there exists in $A$ the supremum of $X$ or, equivalently, the supremum of $\widehat X$ in $[0,1]^{\Max A}$ belongs to $\widehat A$.
\end{definition}

\begin{proposition}\cite[Proposition 4.5]{ciro2016}\label{lc}
Let $A$ be a semisimple MV-algebra. Then a cut $X$ of $A$ is a limit cut if and only if there exists a cut $Y$ of $A$ such that, in $[0,1]^{\Max A}$, $\bigvee \widehat X = \bigwedge \widehat Y^*$, where $Y^* = \{y^* \mid y \in Y\}$. Moreover, $Y$ is a limit cut too.
\end{proposition}

\begin{corollary}\cite[Corollary 4.6]{ciro2016}\label{lccch}
A semisimple MV-algebra $A$ is lcc if and only if, for all $X, Y \subseteq A$ and $\alpha \in [0,1]^{\Max A}$, $\alpha = \bigvee \widehat X = \bigwedge \widehat Y$ implies $\alpha \in \widehat A$.
\end{corollary}
 

 MV-topological spaces have been introduced in \cite{ciro2016} with the aim of extending Stone duality to semisimple MV-algebras. An  MV-topological space is basically a special fuzzy topological space in the sense of C. L. Chang \cite{changfuzzy}. Moreover, most of the concepts and results of the present section are adaptations of the corresponding ones in fuzzy topology. 
\begin{definition}\cite{ciro2016}\label{mvtop}
	Let $X$ be a set, $A$ the MV-algebra $[0,1]^X$ and $\tau \subseteq A$. We say that $( X, \tau )$ is an 
\emph{MV-topological space} (or \emph{MV-space}) if 
	\begin{enumerate}[(i)]
		\item $\0, \1 \in \tau$,
		\item for any family $\{o_i\}_{i \in I}$ of elements of $\tau$, $\bigvee_{i \in I} o_i \in \tau$,
	\end{enumerate}
	and, for all $o_1, o_2 \in \tau$,
	\begin{enumerate}[(i)]
		\setcounter{enumi}{2}
		\item $o_1 \odot o_2 \in \tau$,
		\item $o_1 \oplus o_2 \in \tau$,
		\item $o_1 \wedge o_2 \in \tau$.
	\end{enumerate}
	The set $\tau$ is called an \emph{MV-topology} on $X$ and the elements of $\tau$ are the \emph{open sets} of $X$. The elements of set $\tau^{*} = \{\alpha^* \mid \alpha \in \tau\}$ are called the \emph{closed sets} of $X$. \end{definition}
   
   \begin{definition}\cite{changfuzzy}\label{pre}
Let $X$ and $Y$ be sets and $f: X \lto Y$. The \emph{preimage}, via $f$, of fuzzy subsets of $Y$ is the map defined as
\begin{align*}\label{preim}
f\cou:  [0,1]^Y & \lto  [0,1]^X \\
 \alpha  & \lmapsto  \alpha \circ f.
\end{align*}
and the \emph{image}, via $f$, of fuzzy subsets of X is a map $f^\to :	[0,1]^{X} \lto [0,1]^{Y} $ such that $$f^\to(\alpha)(y)= \bigvee_{f(x) = y} \alpha(x)$$
for  all  $\alpha \in [0,1]^X$ and for all $y \in Y$.\end{definition}
\begin{definition}\cite{changfuzzy}\label{cont}
	Let $(X, \tau_X)$ and $(Y, \tau_Y)$ be two MV-topological spaces. A map $f: X \lto Y$ is said to be
		\begin{itemize}
		\item \emph{continuous} if $f\cou[\tau_Y] \subseteq \tau_X$,
		\item \emph{open} if $f^\to(\alpha) \in \tau_Y$ for all $\alpha \in \tau_X$,
		\item \emph{closed} if $f^\to(\beta) \in \tau^*_Y$ for all $\beta \in \tau^*_X$, and
		\item a \emph{homeomorphism} if it is bijective and both $f$ and $f^{-1}$ are continuous.
	\end{itemize}
\end{definition}

We can now define the category $\TMV$ of MV-topological spaces, as the one whose objects are MV-topological spaces, and having continuous functions among them as morphisms.

\begin{definition}\label{base}
Given a MV-topological space $(X, \tau)$, a subset $B$ of $\tau$ is called a \emph{base} for $\tau$ if every open set of $\tau$ is a join of elements of $B$
\end{definition}
The definition of base for $\tau$ is the same as the one given for fuzzy topological spaces. However, the one of subbases for MV-topologies is different.
\begin{definition}\cite{dlprarc}\label{subbase MV-top}
Given an MV-topological space $(X, \tau)$, a subset $S$ of $\tau$ is called a \emph{subbase} for $(X, \tau)$ if each open set of $X$ can be obtained as a join of finite combinations of products, infima, and sums of elements of $S$. More precisely, $S$ is a subbase for $\tau$ if, for all $\alpha \in \tau$, there exists a family $\{t_i\}_{i \in I}$ of terms (or polynomials) in the language $\{\oplus, \odot, \wedge\}$, such that
\begin{equation}\label{subcomb}
\alpha = \bigvee_{i \in I} t_i(\beta_{i1}, \ldots, \beta_{in_i})
\end{equation}
where, for all $i \in I$, $n_i < \omega$, and $\{\beta_{ij}\}_{j=1}^{n_i} \subseteq S$. A subbase $S$ of an MV-topology $\tau$ shall be called \emph{large} if, for all $\alpha \in S$, $n\alpha \in S$ for all $n < \omega$
\end{definition}
\begin{remark}\label{base-subbase}
If $S$ is a subbase for an MV-topology, the set $B_S$ defined by the following conditions is obviously a base for the same MV-topology:
\begin{enumerate}
  \item [(B1)] $S \subseteq B_S$;
  \item [(B2)] if $\alpha, \beta \in B_S$ then $\alpha \star \beta \in B_S$ for $\star \in \{\oplus, \odot, \wedge\}$.
\end{enumerate}
\end{remark}
\begin{exm}
Let us consider the MV-topological space $(\{x\}, [0,1]^{\{x\}})$ on a singleton $\{x\}$. For any $n > 1$, all the families $[0,1/n]^{\{x\}}$, $([0,1/n] \cap \mathbb{Q})^{\{x\}}$, and $([0,1/n] \setminus \mathbb{Q})^{\{x\}}$ are easily seen to be (non-large) subbases for the given MV-topology. In fact, $([0,1] \setminus \mathbb{Q})^{\{x\}}$ is a base.
\end{exm}

\begin{lemma}\cite{dlprarc}\label{subbascont}
Let $(X, \tau)$ and $(Y, \tau')$ be two MV-topological spaces and let $S$ be a subbase for $\tau'$. A map  $f: X \lto Y$ is continuous if and only if $f\cou[S] \subseteq \tau$.
\end{lemma}

Let $(X, \tau)$ an MV-topological space, a \emph{covering} of $X$ is  defined as any subset $\Gamma$ of $[0,1]^X$ such that $\bigvee \Gamma = \1$. While an \emph{additive covering} ($\oplus$-covering, for short) is a finite subset $\{\alpha_1, \ldots, \alpha_k\}$ of $[0,1]^X$, along with natural numbers $n_1, \ldots, n_k$, such that $n_1\alpha_1 \oplus \cdots \oplus n_k \alpha_k = \1$.
 In the other words, an additive covering  is a finite family $\{\alpha_i\}_{i=1}^n$ of elements of $[0,1]^X$, $n < \omega$, such that $\alpha_1 \oplus \cdots \oplus \alpha_n = \1$.

\begin{definition}\cite{ciro2016}\label{compact}
An MV-topological space $(X, \tau)$ is said to be \emph{compact} if any open covering of $X$ contains an additive covering; it is called \emph{strongly compact}  if any open covering contains a finite covering.\footnote{What we call strong compactness here is called simply compactness in the theory of lattice-valued fuzzy topologies \cite{changfuzzy}.}
\end{definition}
Since $a \vee b \leq a \oplus b$ for any  $a, b \in [0,1]^{X}$ it is immediate to see that strong compactness implies compactness.

\begin{lemma}[Alexander Subbase Lemma for MV-Topologies]\cite{dlprarc}\label{alex}
Let $(X, \tau)$ be an MV-topological space and $S$ a large subbase for $\tau$. If every collection of sets from $S$ that cover $X$ has an additive subcover, then $X$ is compact.
\end{lemma}
\begin{definition}\label{T0}
    An MV-space $(X, \tau)$ is a $T_{0}$ MV-space if for  $x\neq y$ in $X$, there is $\alpha \in \tau$ such that $\alpha(x)\neq \alpha(y)$.
\end{definition}
\begin{definition}\cite{ciro2016}\label{t2ax}
Let $(X, \tau)$ be an MV-topological space. $X$ is called a \emph{Hausdorff} (or \emph{separated}) \emph{space} if, for all $x \neq y \in X$, there exist $\alpha_x, \alpha_y \in \tau$ such that
\begin{enumerate}[(i)]
\item $\alpha_x(x) = \alpha_y(y) = 1$,
\item $\alpha_x \wedge \alpha_y = \0$.
\end{enumerate}
\end{definition}
    

\begin{lemma}\cite[Lemma 3.9]{ciro2016}\label{xclos}
If $(X, \tau)$ is an Hausdorff space, then all crisp singletons of $X$ are closed.
\end{lemma}

\section{MV-lattices and Positive MV-algebras}\label{MV-lattice}

In \cite{abbadini2022}, the authors presented positive MV-algebras as a generalization of MV-algebras, in the same way that distributive lattices generalize Boolean algebras. Here we consider a more general structure that captures the properties of positive MV-algebras without requiring an MV-algebra as a starting point.
\begin{definition}
 Let $(A, +,\cdot, \vee, \wedge, 0,1)$ be an algebra of type $ (2,2,2,2,0,0)$. $A$ is an \emph{MV-lattice} if $(A, \vee, \wedge, 0,1)$ is a bounded distributive lattice, $+$ and $\cdot$ are commutative operations that distribute over $\vee$ and $\wedge$, and satisfy the following equations for all $x,y$ in $A$.
\begin{enumerate}
	\item $x \cdot 1=x$,
	\item $x + 0=x$, and
  \item $(x \cdot y) \vee (x \wedge y)= x\wedge y$.
\end{enumerate}
\end{definition}
\begin{exm}
   Every bounded distributive lattice, with $+:= \vee$ and $\cdot := \wedge$ , is  an MV-lattice. 
     \end{exm}
Given an  MV-algebra $(A, \oplus,^*,0)$ and considering its derived operations, the algebra  $(A, \oplus, \odot, \vee, \wedge,0,1)$ is an MV-lattice. Any subalgebra $ (B, \oplus, \odot, \vee, \wedge,0,1)$ of such an MV-lattice (i.e., any subset of $A$ containing $0$ and $1$ and which is closed under the four binary operations) is called  a \emph{positive subreduct} of the MV-algebra $A$. The following definition is a reformulation of \cite[Definition 2.2]{abbadini2022}.
\begin{definition} An MV-lattice $(A, +, \cdot,\vee, \wedge, 0,1)$ is called a \emph{positive MV-algebra} if it is isomorphic to a positive subreduct of some MV-algebra.
\end{definition}
The class of positive MV-algebras is a quasi-variety generated by the positive standard MV-algebra $ [0,1]$. For more details, the reader may refer to \cite{abbadini2022}.
\begin{exm}
For each natural number n, let $ \textbf{F}_{n}$ be the algebra of McNaughton functions from $[0, 1]^{n}$ to $ [0, 1] $. Then the subset of increasing McNaughton functions, denoted by $ \textbf{F}_{n}^{\leq}$ is a positive subreduct of $\textbf{F}_{n}$. 
\end{exm}
\begin{definition}
A positive MV-algebra $A$ generates an MV-algebra $B$ if the positive subreduct of $B$ that is isomorphic to $A$ is a generating set of $B$.   A positive MV-algebra $A$ that generates an $\lcc$ MV-algebra, is called an \emph{$\lcc$ positive MV-algebra}.
\end{definition}
Obviously, since every finite MV-algebra is $\lcc$, every positive MV-algebra that generates a finite MV-algebra is $\lcc$ too. Moreover, every $\lcc$ MV-algebra is also an $\lcc$ positive MV-algebra.

In the case of MV-algebras  $ \LL_{n+1}$, it was shown in \cite{poiger} that all of their positive MV-subreducts are exactly the MV-subalgebras, i. e., the $\LL_{k+1}$'s, for all divisor $k$ of $n$. 

 \subsection*{Ideals  and Homomorphisms }
 The definition of ideals  of an MV-lattice is identical to the one of MV-algebra ideals.
 \begin{definition}
     An \emph{ideal} of an MV-lattice $(A, +,\cdot, \vee, \wedge, 0,1)$ is a subset $I$ of $A$ satisfying the following conditions:
     \begin{itemize}
         \item[I1)] $0 \in I$,
         \item[I2)] If $a \in I, b\in A$ and $b \leq a$ then $ b\in I$, and
         \item[I3)] If $x\in I$ and $y \in I$ then $a+b \in I$.
     \end{itemize}
     If in addition
     \begin{itemize}
         \item[I4)] $a \wedge b \in I$ implies that $a \in I$ or $b \in I$,
     \end{itemize}
    then $I$ is called a \emph{prime}.
 \end{definition}
 If $A$ is a positive subreduct of an MV-algebra $B$,  let us consider the set \begin{center}
         $\Max _{\leq}A= \{ I \cap A : I \in \Max B\} $
     \end{center}   ordered  by inclusion. For each prime ideal $I$ of $B$,  $I \cap A$ is a prime ideal of $A$. However, the elements of $ \Max_{\leq}A$ are  prime, but not necessarily  maximal ideals of $A$. 
\begin{exm}\label{L2} The set $A=\left\{(0,0), \left(0, \frac{1}{2} \right ),(0,1), \left(\frac{1}{2}, 1\right), (1,1)\right\} \subset \LL_{3}\times \LL_{3}$ is a generating positive MV-algebra of $\LL_{3} \times \LL_{3}$
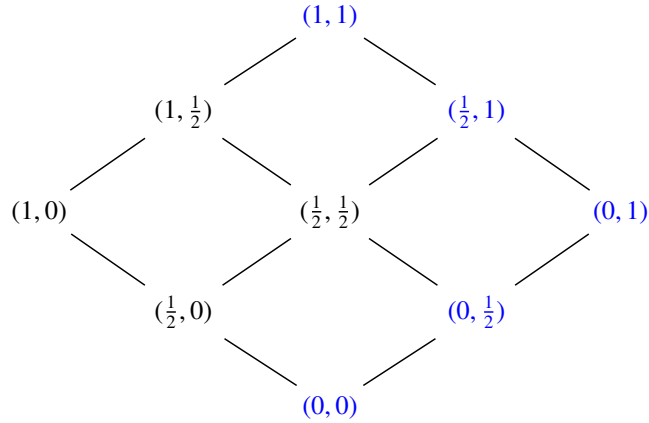
\begin{center}
\begin{tikzcd}
                            &                                                           & \textcolor{blue}{{(1,1)}}                                                             &                                                            &                             \\
                            & {(1, \frac{1}{2})} \arrow[ru, no head]                    &                                                                      & \textcolor{blue}{(\frac{1}{2},1)} \arrow[lu, no head]                      &                             \\
{(1,0)} \arrow[ru, no head] &                                                           & (\frac{1}{2} ,\frac{1}{2}) \arrow[ru, no head] \arrow[lu, no head] &                                                            & \textcolor{blue}{{(0,1)}} \arrow[lu, no head] \\
                            & {(\frac{1}{2},0)} \arrow[ru, no head] \arrow[lu, no head] &                                                                      & \textcolor{blue}{{(0, \frac{1}{2})}}\arrow[lu, no head] \arrow[ru, no head] &                             \\
                            &                                                           & \textcolor{blue}{{(0,0)}} \arrow[ru, no head] \arrow[lu, no head]                      &                                                            &                            
\end{tikzcd}
\end{center}
     
where, the  maximal ideals of $ \LL_{3} \times \LL_{3}$ are $I_{1}=\{(0,0), (0, \frac{1}{2}),(0,1)\}$ and $I_{2}=\{(0,0), (\frac{1}{2},0),$ $ (1,0)\}$. Then, $\Max_{\leq}A=\{I_{1}^{A}, I_{2}^{A}\}$ where $I_{2}^{A}=\{(0,0)\} \subset I_{1}=I_{1}^{A}$. So, $I_{2}^{A}$ is not a maximal ideal of $A$.
\end{exm}
\begin{definition} Let 
	A and B be MV-lattices. A map $f:A \lto B$ is an \emph{MV-lattice homomorphism} iff  $f$ is a homomorphism of bounded lattices which preserves the operations $+$ and $\cdot$. 
\end{definition}
The category of MV-lattices and its morphisms will be denoted by $\MVL$, and we will denote by $\MV_{+}$ and $\MVlcc_{+}$ its full subcategories of positive MV-algebras  and $\lcc$ positive MV-algebras, respectively.
Given an MV-lattice $(A, +,\cdot, \vee, \wedge, 0,1)$, we consider the set of morphisms from $A$ to $[0,1]$, which is denoted by $H_{A}$ and is equipped with a partial order defined by $ f \leq g $ if and only if $ f(a) \leq g(a) $ for every $ a $ in $ A $. For $A$ as in Example \ref{L2}, we have that $H_{A}=\{f_{1}, f_{2}\} $ where $ f_{i}: A \lto \LL_{3} \subseteq [0,1] $ is defined by $f_{i}(x_{1},x_{2})=x_{i}$ for $(x_{1},x_{2}) \in A$ and $f_{1}\leq f_{2}$. Note that $f_{1}$ and $f_{2}$ are restrictions of homomorphisms from the MV-algebra $ \LL_{3}\times \LL_{3}$ to $[0,1]$. This is actually an instance of a general situation, as shown in the following lemma.

\begin{lemma}\label{fextension}
Let $B$ be an MV-algebra and  $A$ be a positive MV-subreduct of $B$ that generates $B$. Then, for any $f \in \operatorname{Hom}_{\MV_{+}}(A, [0,1]) $, there exists a unique  $\bar{f} $  such that $ \bar{f} \in \operatorname{Hom}_{\text{MV}}(B, [0,1]) $  and $ \bar{f} \mid_A = f $.
\end{lemma}
\begin{proof}
Since $ A $ generates $ B $, from Proposition \ref{termo}, each $b \in B $ can be expressed as
$b = t(a_{1}, \ldots, a_{s})$
where $ t $ is a term in the language $\{\oplus, ^{*}\} $ and $ a_i \in A $.

Defining $ \bar{f}(b) = t(f(a_{1}), \ldots, f(a_{s})) $, we have that $ \bar{f} $ preserves $\oplus $ and $^{*} $, and $ \bar{f} \mid_A = f $. However, the expression $ b = t(a_{1}, \ldots, a_{s}) $  in terms of elements of $A$ may not be unique. Let us prove that $\bar{f}$ is well defined. Let $ b = t(a_{1}, \ldots, a_{s}) = q(c_{1}, \ldots, c_{r}) $, with $ \bar{a}=(a_{1}, \ldots, a_{s})$, and $ \bar{c} = (c_{1}, \ldots, c_{r})$, we have
\[
\begin{aligned}
0 &= f(0) = \bar{f}(0) = \bar{f}(d ( t(\bar{a}), \bar{q}(\bar{c})) \\
&= d(t(f(a_{1}), \ldots, f(a_{s})), q(f(c_{1}), \ldots, f(c_{r}))),
\end{aligned}
\]
where $d$ is the distance on an MV-algebra, defined by 
$d(a,b)=(a \ominus b)\oplus (b \ominus a) $. Thus, $ \bar{f}(t(\bar{a})) = t(f(a_{1}), \ldots, f(a_{s})) = q(f(c_{1}), \ldots, f(c_{r})) = \bar{f}(q(\bar{c})) $, and $ \bar{f} $ is well-defined. The uniqueness of $ \bar{f} $ follows from the fact that $ A $ generates $ B $.
\end{proof}
The  previous proof  relies on elementary concepts but the result holds for any MV-algebra $C$ in place of the interval $[0,1]$. In  \cite{abbadini2022}, this generalization is established by using the equivalence between MV-monoidal algebras and unital commutative $\ell$-monoids.

The following result provides a necessary and sufficient condition for an MV-lattice $A$ to be embedded in a semisimple MV-algebra. We recall that maximal ideals of any MV-algebra are in bijective correspondence with homomorphisms from the given algebra to $[0,1]$. Therefore, an MV-algebra $B$ is semisimple if and only if the intersection of the kernels of all the homomorphisms from $B$ to $[0,1]$ is the identity congruence $\Delta$.

\begin{proposition}\label{A}
Let $ A $ be an MV-lattice. $ A $ is a positive subreduct  generating a semisimple MV-algebra $ B $ if and only if $ \bigcap\limits_{f \in H_{A} } \ker f = \Delta $.
\end{proposition}
\begin{proof}
Let $H_{A}$ be the set of  homomorphisms  $f:A \lto [0,1]$
 and  $\tilde{A}=\{\tilde{a} \mid a\in A\}$, where 
\begin{align*}
	\tilde{a}:H_{A} & \lto [0,1]\\
	f & \mapsto f(a).
\end{align*}
For all $f \in H_A$, consider $\ker f=\{(a,b)\in A \times A \mid f(a)=f(b) \}$, and the map
\begin{align*}
    T:A &\lto \tilde{A}\\
    a& \mapsto \tilde{a}.
\end{align*}
If $\bigcap\limits_{f \in H_{A}}\ker f= \Delta$, $T$ is an isomorphism of MV-lattices and
$ \tilde{A} \subset [0,1]^{H_{A}} $ generates a semisimple MV-algebra $ B = \langle\tilde{A}\rangle \subseteq [0,1]^{H_{A}} $.

Conversely, if $A$ is a positive subreduct of a semisimple MV-algebra $B$ and  $A$  generates $B$, by Lemma \ref{fextension}, $\left|H_{A}\right| = \left|\operatorname{Hom}_{\MV}(B, [0,1])\right| $ and, since $ B $ is semisimple, we have that

\begin{center}
$\{0\} = \operatorname{Rad} B = \bigcap\limits_{f \in \operatorname{Hom}_{\MV}(B, [0,1])} \bar{f}^{-1}[0] \supseteq \bigcap\limits_{f \in H_{A}} \bar{f}^{-1}[0].$
\end{center}
 Let  $f \in H_{A}$. If $ f(a) = f(b) $ for all $ f \in H_{A} $, then $$0=d(f(a), f(b))=d(\bar{f}(a), \bar{f}(b))=\bar{f}(d(a, b)).$$
 So, $d(a,b)\in \bigcap\limits_{\bar{f} \in H_{A}} f^{-1}[0] = \{0\}$, hence  $a = b$ and thus $\bigcap\limits_{f \in H_{A}}\ker f= \Delta$.
\end{proof}

It is known that an MV-algebra  can be generated by various positive algebras. According to Lemma \ref{fextension}, the corresponding sets of homomorphisms from these algebras into $[0,1]$ are equipotent. However, these sets are not necessarily isomorphic partially ordered sets, as the following example shows.
\begin{exm}\label{L223} Given the MV-algebra $E=\LL_{3}\times \LL_{3}\times \LL_{3}$, we have that the set\\ $\operatorname{Hom}_{\MV}(E, [0,1])=\{f_{1}, f_{2}, f_{3}\}$ where $f_{i}(x_{1}, x_{2}, x_{3})=x_{i}$. So, $ \Max E=\{I_{1}, I_{2}, I_{3}\}$ where $I_{i}=f_{i}^{-1}[0]=\{(x_{1},x_{2},x_{3}) \in E \mid x_{i}=0\}$. Moreover, by Lemma \ref{fextension}  for any positive subreduct $A$ that generates $E$, the elements of $H_{A}$ are restrictions $f_{i}|_{A}$ that we will denote by $f_{i_{A}}$. Now we consider three different generating positive MV-algebras of $E$:
\begin{enumerate}
    \item The positive MV-algebra
		$$\begin{array}{ll}
		A= & \left\{ \left(0,0,0\right), \left(0, \frac{1}{2},0\right), \left(0,1,0\right), \left(\frac{1}{2},1,0\right), \left(1,1,0\right), \right.\\
		 & \left. \left(0,\frac{1}{2},\frac{1}{3}\right), \left(0, 1,\frac{1}{3}\right),  \left(\frac{1}{2}, 1, \frac{1}{3}\right), \left(1, 1, \frac{1}{3}\right), \left(0,1,\frac{2}{3}\right), \right. \\
		 & \left. \left(\frac{1}{2},1,\frac{2}{3}\right), \left(1,1,\frac{2}{3}\right), \left(0,1,1\right), \left(\frac{1}{2},1,1\right),  \left(1,1,1\right)\right\}
		\end{array}$$ has $H_{A}=\{f_{1_{A}}, f_{2_{A}}, f_{3_{A}}\}$ with $f_{1_{A}} \leq f_{2_{A}}$, $f_{3_{A}} \leq f_{2_{A}}$, and $f_{1_{A}}$ and$f_{3_{A}}$ uncomparable.  Moreover $\Max_{\leq}A=\{I_{1}^{A}, I_{2}^{A}, I_{3}^{A}\}$ with $I_{1}=\{(0,0,0), (0,\frac{1}{2},0) $, $(0,1,0), $$ (0, \frac{1}{2}, \frac{1}{3}) $, $(0,1,\frac{1}{3}),$ $ (0,1, \frac{2}{3}), \ (0,1,1)\}$, $I_{2}=\{(0,0,0)\} $, and  $I_{3}=\{(0,0,0),$ $ (0,\frac{1}{2},0), (0,1,0),$ $ (\frac{1}{2},1,0),$ $ (1,1,0)\}$. So, $H_{A}$ and $\Max_{\leq}A$ have the following Hasse diagrams 
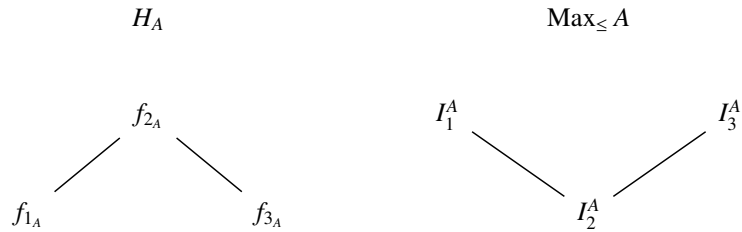
\begin{center}
   \begin{tikzcd}
                              & H_{A}                         &           &  &                               & \Max_{\leq}A &                               \\
                              & f_{2_{A}} \arrow[rd, no head] &           &  & I_{1}^{A} \arrow[rd, no head] &               & I_{3}^{A} \arrow[ld, no head] \\
f_{1_{A}} \arrow[ru, no head] &                               & f_{3_{A}} &  &                               & I_{2}^{A}     &                              
\end{tikzcd}
\end{center}
\item For the positive MV-algebra $B=A \cup \left\{\left(\frac{1}{2}, \frac{1}{2},0\right),\ \left(\frac{1}{2}, \frac{1}{2}, \frac{1}{3}\right)\right\}$, we have that $I_{1}^{B}=I_{1}^{A}$, $I_{2}^{B}=I_{2}^{A}$ and $I_{3}^{B}= I_{3}^{A} \cup  \left\{\left(\frac{1}{2}, \frac{1}{2},0\right)\right\}$. Moreover, since that $f_{2}\left(\frac{1}{2}, \frac{1}{2},0\right) \geq f_{1}\left(\frac{1}{2}, \frac{1}{2},0\right)$  and $f_{2}\left(\frac{1}{2}, \frac{1}{2},\frac{1}{3}\right) \geq f_{1}\left(\frac{1}{2}, \frac{1}{2},\frac{1}{3}\right)$, then $H_{B}$ and $H_{A}$ are order isomorphic sets, as well as $\Max _{\leq }A$ and  $\Max _{\leq }B$. 
 \item  The positive MV-algebra
 $$C= B \cup \left\{\left(0,0, \frac{1}{3}\right), \left(0,0, \frac{2}{3}\right), \left(\frac{1}{2}, \frac{1}{2}, \frac{2}{3}\right), \left(0,\frac{1}{2},\frac{2}{3}\right), \left(0,0,1\right), \left(0, \frac{1}{2},1\right), \left(\frac{1}{2}, \frac{1}{2},1\right)\right\}$$
 is such that
$$\begin{array}{l} I_{1}^{C}=I_{1}, \\
I_{2}^{C}=\left\{\left(0,0,0\right), \left(0,0, \frac{1}{3}\right), \left(0,0, \frac{2}{3}\right), (0,0,1)\right\}, \text{ and} \\
I_{3}^{C}=\left\{(0,0,0), \left(0,\frac{1}{2},0\right), (0,1,0), \left(\frac{1}{2}, \frac{1}{2},0\right), \left(\frac{1}{2},1,0\right), (1,1,0)\right\}.
\end{array}$$
Moreover $f_{1_{C}} \leq f_{2_{C}}$ but $f_{2_{C}}$ and $f_{3_{C}}$ are not comparable, since
$$f_{3}\left(\frac{1}{2}, \frac{1}{2}, \frac{2}{3}\right)=\frac{2}{3}>\frac{1}{2}=f_{2}\left(\frac{1}{2}, \frac{1}{2}, \frac{2}{3}\right) \text{ and } f_{3}\left(\frac{1}{2}, \frac{1}{2},0\right)=0 < \frac{1}{2}=f_{2}\left(\frac{1}{2}, \frac{1}{2},0\right).
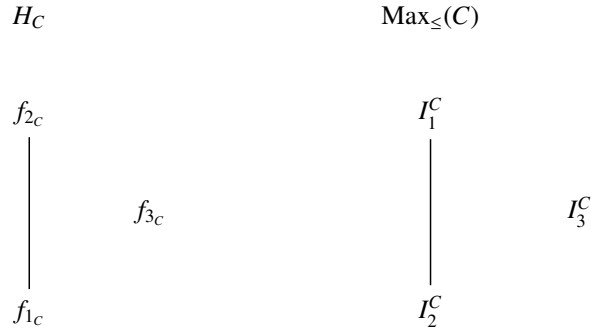
$$ 
    
$$\begin{tikzcd}
H_{C}                         &           &  &  & \text{Max}_{\leq}(C)          &           \\
f_{2_{C}} \arrow[dd, no head] &           &  &  & I_{1}^{C} \arrow[dd, no head] &           \\
                              & f_{3_{C}} &  &  &                               & I_{3}^{C} \\
f_{1_{C}}                     &           &  &  & I_{2}^{C}                     &          
\end{tikzcd}$$
 So, $H_{A}$ and $H_{C}$ are equipotent sets but are not order isomorphic.
\end{enumerate}
\end{exm}
We observe that the sets of morphism $H_A$ and $H_B$ are isomorphic, which is not the case for all generating algebras, as exemplified by $C$. This situation motivates to introduce a definition which includes those generating algebras having this property.
\begin{definition}\label{R_A}
    Let $A$ be a generating positive MV-algebra of an MV-algebra $B$. We define
    \begin{center}
        $\preccurlyeq_{A}=\{(f_{i}, f_{j}) \in \operatorname{Hom}_{\MV}(B, [0,1])^{2} \mid f_{i_{A}} \leq f_{j_{A}}\}$
    \end{center}
    and, if $A'$ is  a generating positive MV-algebra of $B$ such that $\preccurlyeq_{A'}=\preccurlyeq_{A}$, we will say that $A$ and $A'$ are \textit{compatible.}
    
    In Example \ref{L223}, $A$ and $B$ are compatible, but $C$ is not compatible with either of them.
\end{definition}
\begin{remark}
Note that the set $\preccurlyeq_{A}$ in the previous definition collects the arrows from the Hasse diagram of $H_{A}$. For two generating positive MV-algebras of $B$, $A$ and $C$, if $A \subseteq C$, then $\preccurlyeq_{C} \subseteq \preccurlyeq_{A}$ and \begin{align*}
         T: \Max_{\leq}C &\lto \Max_{\leq}A\\
         I_{i}^{C}& \longmapsto \ I_{i}^{A}
     \end{align*}
     is an order-preserving bijective function. Moreover, $\preccurlyeq_{B}$ is always the identity relation. Indeed, if $(f,g) \in \preccurlyeq_{B}$ for some $f,g \in \operatorname{Hom}_{\MV}(B, [0,1])$, then $g^{-1}[0] \subseteq f^{-1}[0]$ and they are both maximal ideals of $B$, therefore $g^{-1}[0] = f^{-1}[0]$, whence $f=g$.
\end{remark}
   \begin{definition}
Let $A$ be a positive MV-algebra that generates an MV-algebra $B$. $A$ is called $H$-\emph{complete} if, for each generating positive MV-algebra $E$, the following holds 
$$A \subseteq E \text{ and } \preccurlyeq_{A} = \preccurlyeq_{E} \text{ imply } A=E.$$
   \end{definition}
In Example \ref{L223}, $B$ and $C$ are $H$-complete, while $A$ is not. 

On the other hand, we note that there is a connection between the edges of the Hasse diagrams of $H_{A}$ and $\Max_{\leq}A$.  It is clear that if $f_{i_{A}} \leq f_{j_{A}}$, then $I_{j}^{A}=f_{j_{A}}^{-1}[0] \subseteq f_{i_{A}}^{-1}[0]=I_{i}^{A}$. Is the converse also true? The answer is affirmative, and to prove this, we need the following result.
\begin{lemma}\label{separacion}If $x$ and $y$ are in the interval 
$[0,1]$  and $x<y $, then there exists a term 
$t$ in the language $\{\oplus, \odot\}$ such that 
$t(x)=0$ and $t(y)=1$.
\end{lemma}
\begin{proof}
First, let us consider the following two steps for $x, y \in [0,1]$:

\underline{	Step 1}. If $x<y<\frac{1}{2}$ we consider $k=\max\left\{n \in \N \mid \ y \leq \frac{1}{n}\right\}$. Then $k \geq 2$, $k  x < k  y \leq 1$ and consequently $$k y \ominus k  x= ky-kx=k(y-x)\geq 2(y-x)> y \ominus x$$ with $k  y >\frac{1}{2}$.

\underline{	Step 2}. If $\frac{1}{2}<x<y$ we consider $h=\max\left\{ n \in \N \mid x > \frac{n-1}{n} \right\}$ and $x^{h}$ as  the product $\odot$ of $x$ with itself n times. Then $h \geq 2$ and $x^{h}= \max\{0, hx-h+1\}=hx-h+1$. Moreover $$y^{h} \ominus x^{h}=y^{h}-x^{h}=h(y-x)=h(y \ominus x) \geq 2(y\ominus x) > y\ominus x.$$
	 Now we consider the following cases:
\begin{enumerate}
	\item If $x\leq \frac{1}{2} < y$, then  $x^{2}=0$, $y^{2} \neq 0$, and there exists $n \in \N$ such that $n y^{2}=1$ and $n x^{2}=0$. That is, there is a term defined by $t(a)=n (a)^{2}$ for $ a \in [0,1]$ such that $t(y)=1$ and $t(x)=0$.
	\item If $x<\frac{1}{2}\leq y$, \ $2x<1=2y$ and, for some $m \in \N$, $(2x)^{m}=0$ and $(2y)^{m}=1$. So, the term defined by $t(a)=(2a)^{m}$  for $a \in [0,1]$ is such that $t(y)=1$ and $t(x)=0$.
	\item If $x<y<\frac{1}{2}$, we apply Step 1  and we obtain that $y_{1} \ominus x_{1} \geq 2(y\ominus x)>y \ominus x$ where $y_{1}=ky$ and $x_{1}=kx$. Moreover, since  $y_{1}>\frac{1}{2}$, then either $x_{1}\leq \frac{1}{2}$ or $x_{1}>\frac{1}{2}$. If $x_{1} \leq \frac{1}{2}$, by case 1, there is $n_{1} \in \N$ such that $n_{1}x_{1}^{2}=0$ and $n_{1}y_{1}^{2}=1$. That is, there is a term $t(a)=n_{1}(ka)^{2}$ such that $t(x)=0$ and $t(y)=1$. If $x_{1}>\frac{1}{2}$, we apply Step 2 and we obtain
	$$y_{2} \ominus x_{2}=y_{1}^{h}-x_{1}^{h}=h(y_{1} \ominus x_{1}) \geq 2(y_{1}\ominus x_{1})> y_{1} \ominus x_{1}$$ where $y_{2}=y_{1}^{h}$ and $x_{2}=x_{1}^{h}$ 
 . Note that $x_{2} < \frac{1}{2}$ and again we can apply Step 1 or case 2 according to the possible values of $y_{2}$. Thus we define by recursion a sequence of pairs $(x_{i},y_{i})$ such that  for each $i \geq 1$, \ $y_{i+1}\ominus x_{i+1}\geq 2(y_{i} \ominus x_{i})> y_{i} \ominus x_{i}$. For this reason, for some $n \in \N$, $x_{n}$ and $y_{n}$ will satisfy the conditions of case 1 or 2 and we shall obtain a term $t$ such that $t(x)=0$ and $t(y)=1$. 
	\item If $\frac{1}{2}< x<y$, we apply Step 2 as in case 3 and we get to the same conclusion. 
\end{enumerate}	
\end{proof}

  \begin{proposition}\label{hmax}
      Let $A$ be a generating positive MV-algebra of an MV-algebra $B$ and $f_{i}, f_{j} \in H_{A}$. Then $f_{i} \leq f_{j}$ if and only if $f_{j}^{-1}[0] \subseteq f_{i}^{-1}[0]$. 
  \end{proposition}
  \begin{proof} 
The necessity is clear. To prove sufficiency, we proceed by contrapositive, assuming $f_{i}, f_{j} \in H_{A}$ such that $f_{i} \nleq f_{j}$. Then $f_{j}(a) < f_{i}(a)$ for some $a \in A$ and, applying the previous lemma, we have that there exists a term $t$ in the language $\{ \oplus, \odot\}$ such that $t(f_{j}(a))=0$ and $t(f_{i}(a))=1$. Since $f_{i}$ and $f_{j}$ preserve $\oplus $ and $\odot$, $0=t(f_{j}(a))=f_{j}(t(a))$ and $1=t(f_{i}(a))=f_{j}(t(a))$. So, $f_{j}^{-1}[0] \not\subseteq f_{i}^{-1}[0]$, because there is $t(a) \in A$ such that $t(a) \in f_{j}^{-1}[0]$ and $t(a)\notin f_{i}^{-1}[0]$.
  \end{proof} 
The result above allows us to suitably transition from $H_{A}$ to $\Max_{\leq}A$ whenever convenient.
  \begin{corollary}\label{corord}
      Let $A$ and $E$ be two generating positive MV-algebras of an MV-algebra $B$. Then $A$ and $E$ are compatible if and only if $Max_{\leq}A$ and $Max_{\leq}E$ are isomorphic.
  \end{corollary}
  \begin{proof}
      We have that $\preccurlyeq_{A}=\preccurlyeq_{E}$ if and only if, for $f_{i}, f_{j}\in \operatorname{Hom}_{\MV}(B,[0,1])$,  the equivalence $f_{i_{A}} \leq f_{j_{A}} \Leftrightarrow f_{i_{E}} \leq f_{j_{E}}$ holds. By Proposition \ref{hmax}, this is equivalent to $I_{j}^{A} \subseteq I_{i}^{A} \Leftrightarrow I_{j}^{E} \subseteq I_{i}^{E}$, that is, the map
      \begin{align*}
         T: \Max_{\leq}E &\lto \Max_{\leq}A\\
         I_{i}^{E}& \longmapsto I_{i}^{A}
     \end{align*}
     is an isomorphism of partially ordered sets.
  \end{proof}
\begin{corollary}
    Let $A$ be a generating positive MV-algebra of an MV-algebra $B$. The following are equivalent:
    \begin{enumerate}
    \item[(a)] $A$ is $H$-complete;
    \item[(b)] for each generating positive MV-algebra $E$ of $B$, if $A \subseteq E$ and $T: I_{i}^{E} \in \Max_{\leq}E \longmapsto I_{i}^{A} \in \Max_{\leq}A$ is an isomorphism, then $A=E$.
    \end{enumerate}
\end{corollary}
\begin{proof}
    The result follows from  Definition \ref{R_A} and the fact that, according to a Corollary \ref{corord}, $\preccurlyeq_{A}=\preccurlyeq_{E}$ if and only if $T$ is an isomorphism.
\end{proof}

\section{The extension of Priestley Duality}\label{dualsec}
In this section we shall prove that Priestley Duality can be extended to a category of positive MV-algebras and one of ordered compact MV-topological spaces having a base of clopens. For this we intruduce the \textit{partially ordered MV-spaces} as triples $(X, \tau , \leq)$, where $(X, \tau)$ is an MV-space and $\leq$ is a partial order on $X$. Moreover, we present the definition of Priestley spaces in the context of MV-topology as follows.
 \begin{definition}	 A partially ordered MV-space $(X, \tau, \leq)$ is a \textit{Priestley MV-space} if
	 \begin{itemize}
	 	\item[P1.] $(X, \tau)$ is compact,
	 	\item[P2.] $x \nleq y$ implies that there is an increasing clopen set $\alpha$ such that $\alpha(x)=1$ and $\alpha(y)=0$, and
   \item[P3.] $(X, \tau)$ has a basis of clopens.
	 \end{itemize}
	\end{definition}
    MV-spaces satisfying P2 are called  totally order disconnected. It is easy to verify that every totally order disconected MV-space is a Hausdorff MV-space.
    We will denote by $\PMV$ the category of partially ordered MV-spaces whose morphisms are functions that are simultaneously continuous and order-preserving. Moreover, we shall denote by  $\Pries$ the full subcategory of $\PMV$ whose objects are Priestley MV-spaces.

When $A$ is an MV-algebra, it is possible to obtain an MV-topological space by endowing the set of maximal ideals of $A$ with an MV-topology, as in \cite{ciro2016}. In order to associate a topology to a distributive lattice $A$, Priestley used the set of homomorphisms from $A$ to the two-element distributive lattice $\{0,1\}$ \cite{pri70}. Here, we will follow Priestley's ideas by taking $[0,1]$ instead $\{0,1\}$ and topologizing the set of morphisms from an MV-lattice $A$ to $[0,1]$.\\
Let $(A, +,\cdot, \vee, \wedge, 0,1)$ be an MV-lattice and $H_{A} $ be the set of  homomorphisms  $f:A \lto [0,1]$. $H_{A}$ has a partial order  defined by $f \leq g$ iff $f(a)\leq g(a)$ for all $a \in A$. Now, in order to define an MV-topology on $H_{A}$, let us consider the set $\tilde{A}=\{\tilde{a} \mid a\in A\}$ as in Proposition \ref{A}, where 
$$\begin{array}{cccc}
	\tilde{a}: & H_{A} & \lto & [0,1]\\
	& f & \mapsto & f(a)
\end{array}$$
The set $\tilde{A}$ generates a semisimple MV-algebra $B \subseteq [0,1]^{H_{A}}$ which, on its turn, is a basis for an MV-topology $\tau_{A}$ on $H_{A}$ (see \cite[Section 4]{ciro2016}). This partially ordered MV-space $(H_{A}, \tau_{A}, \leq)$  generated by $B$ is called the  \textit{MV-space associated to $A$}.
\begin{lemma}\label{equipotencia}
Let	 $(H_{A}, \tau, \leq)$  be the MV-space associated with the MV-lattice $A$	and $B$ the MV-algebra generated by $\tilde{A}$. Then $\left|\Max B\right|=\left|H_{A}\right|$.
\end{lemma}
\begin{proof}
	To check that $\left|\Max B\right|=\left|H_{A}\right|$, notice that $\left|\Max B\right|=\left|\hom_{\MV}(B, [0,1])\right|$. So, to  show that $\left|\hom_{\MV}(B, [0,1])\right|=\left|H_{A
	}\right|$ we define a map \begin{align*}
		\varphi: \hom_{\MV}(B, [0,1])&\lto H_{A}\\
		\alpha & \mapsto g_{\alpha}   
	\end{align*}
	where for each $\alpha \in \hom_{\MV}(B, [0,1])$, the map $g_{\alpha} \in H_{A}$ is defined by $g_{\alpha}(a)= \alpha(\tilde{a})$ for all $a \in A$. Besides for $\alpha, \beta \in \hom_{\MV}(B, [0,1]) $, $\varphi(\alpha)= \varphi(\beta)$ implies that $\alpha(\tilde{a})= \beta(\tilde{a})$ for all $\tilde{a} \in \tilde{A}$ and since $B= \langle \tilde{A} \rangle$, we have that $\alpha(x)=\beta(x)$ for all $x \in B$ and $\varphi$ is injective.
    
	On the other hand, 	since $B$ is the MV-algebra generated by $\tilde{A}$, for each $b$ in $B$, $b=p(a_{1}, \ldots, a_{n})$ where $p$ is a term in the language $\{ \oplus, ^{*}\}$ and  $a_{i} \in \tilde{A}$ for all $i$. Moreover, for $f\in H_{A}$,
	\begin{align*}
		b(f)=&p(a_{1}, \dots, a_{n})(f)\\
		=& p(a_{1}(f), \dots, a_{n}(f))\\
		=& p(f(a_{1}), \cdots, f(a_{n})).
	\end{align*}
	So, for each $f\in H_{A}$ there is a map $\bar{f}:b \in B \mapsto b(f) \in [0,1]$ which is a homomorphism of MV-algebras and the map
	\begin{align*}
		\psi: H_{A} &\lto \hom_{\MV}(B,[0,1])\\
               f	& \mapsto \bar{f} 
\end{align*}
	 is such that $\varphi \circ \psi(f)=\varphi(\bar{f})=g_{\bar{f}}$ and  $g_{\bar{f}}(a)=\bar{f}(\tilde{a})=\tilde{a}(f)=f(a)$ for all $a \in A$. That is $g_{\bar{f}}=f$ and $\varphi \circ \psi= \id_{H_{A}}$. Moreover, for $\alpha \in \hom_{\MV}(B, [0,1])$, $\psi \circ \varphi (\alpha)=\bar{g_{\alpha}}$, where $\bar{g_{\alpha}}(\tilde{a})=\tilde{a}(g_{\alpha})=g_{\alpha}(a)=\alpha(\tilde{a})$ for all $\tilde{a} \in \tilde{A}$. Then, since $B=\langle \tilde{A} \rangle$, we have that $\alpha= \bar{g_{\alpha}}$, i.e. $\psi \circ \varphi= \id_{\hom_{\MV}(B, [0,1])}$. Therefore $\left|\hom_{\MV}(B, [0,1])\right|=\left|H_{A}\right|$ and consequently $\left|\Max B\right|=\left|H_{A}\right|$.
\end{proof}
\begin{proposition}
     If $A$ and $E$ are MV-lattices that generate the same algebra $B$ up to an isomorphism, then the ordered MV-spaces $(H_{A}, \tau_{A}, \leq)$ and $(H_{E}, \tau_{E}, \leq)$ are order-homeomorphic if and only if $A$ and $E$ are compatible up to an isomorphism.
 \end{proposition}
\begin{proof}
    Let $A$ and $B$ be MV-lattices. We consider the MV-algebras $B_{A}=\langle \tilde{A} \rangle$ and $B_{E}=\langle \tilde{E} \rangle$ and suppose that there exists an isomorphism $k: B_{A} \to B_{E}$. Then, the function $t: \hom_{\MV}(B_{E}, [0,1]) \to \hom_{\MV}(B_{A}, [0,1])$, defined by
		$t(f)=f \circ k$ for $f\in \hom_{\MV}(B_{E}, [0,1]) $ is a bijective map. By Lemma \ref{equipotencia} there exist bijections between $\hom_{\MV}(B_{A}, [0,1])$ and $|H_{A}|$,  and  between $|H_{E}|$ and $\hom_{\MV}(B_{E}, [0,1])$. Then there exists a bijection between $H_{A}$ and $H_{E}$ namely, the map $T = \varphi_{A} \circ t \circ \psi_{E}: H_{E} \to H_{A}$ where, according to Lemma \ref{equipotencia},
    \begin{enumerate}
    \item $\psi_{E}: H_{E} \lto \hom_{\MV}(B_{E},[0,1])$ is defined by $\psi_{E}(f)(b)=b(f)$ for $b \in B_{E}$ and $f \in H_{E}$, and 
\item $\varphi_{A}: \hom_{\MV}(B_{A}, [0,1])\lto H_{A}$ 
	is defined, for each $g \in \hom_{\MV}(B_{A}, [0,1])$, by $\varphi_A(g) = g_{A} \in H_{A}$, and $g_{A}: a \in A \to g(\tilde{a}) \in [0,1]$ for all $a \in A$.
\end{enumerate}
Note that the MV-spaces $(H_{A}, \tau_{A})$ and $(H_{E}, \tau_{E})$ are homeomorphic because $\tau_{A}$ and $\tau_{E}$ have the same basis (up to $t$). Indeed, $T^{\cou}: [0,1]^{H_{A}} \to [0,1]^{H_{E}}$ acts, for all $b \in B_{A}$ and $f_{E} \in H_{E}$, as follows:
\begin{align*}
    T^{\cou}(b)(f_{E})=& b\circ T(f_{E})\\
    =& \psi_{A}(T(f_{E}))(b)\\
    =& \psi_{A} \circ \varphi_{A} \circ t\circ \psi_{E}(f_{E})(b)\\
    =& t\circ \psi_{E}(f_{E})(b)\\
    =&(\psi_{E}(f_{E}) \circ k)(b)\\
    =& \psi(f_{E})\circ k(b)\\
    =&\psi(f_{E})(k(b))\\
    =& k(b)(f_{E}).
\end{align*}
So $T^{\cou}(b)=k(b)\in B_{A}$ and therefore  $T$ is continuous. Analogously, it is easy to see that $T^{-1}$ is continuous, whence $(H_{A},  \tau_{A} )$ and $(H_{E}, \tau_{E})$ are homeomorphic MV-spaces.

Now, let us consider also the order relations and suppose that $\tilde{E}$ and $k(\tilde{A})$ are compatible positive MV-algebras.  If $f, g \in H_{E}$ satisfy $f\leq g $, then $\tilde{x}(f)=f(x)\leq g(x)= \tilde{x} (g)$ for all $\tilde{x} \in \tilde{E}$. Since $\psi_{E}(f)(\tilde{x})=\tilde{x}(f)$, then $\psi_{E}(f)(\tilde{x}) \leq \psi_{E}(g)(\tilde{x})$ for $\tilde{x} \in \tilde{E}$, that is, $(\psi_{E}(f), \psi_{E}(g)) \in \preccurlyeq_{E}$. Now, since $\preccurlyeq_{E}=\preccurlyeq_{k(\tilde{A})}$, it follows that  $\psi_{E}(f)(k(\tilde{a}) \leq \psi_{E}(g)(k(\tilde{a}))$ for all $\tilde{a} \in \tilde{A}$ and, therefore, for $a \in A$, we have
\begin{align*}
    T(f)(a)=& (\varphi_{A}\circ t \circ \psi_{E})(f)(a)\\
    =& \varphi_{A}(\psi_{E}(f)\circ k)(a)\\
    =&(\psi_{E} \circ k)(\tilde{a})\\
    =& \psi_{E}(f)(k(\tilde{a}))\\
    \leq & \psi_{E}(g)(k(\tilde{a}))\\
    =&T(g)(a).
\end{align*}
Then $T$  is order-preserving. Again, the proof of the same property for $T^{-1}$ is completely analogous.

In order to prove that the compatibility of $A$ and $E$ is also a sufficient condition, we proceed by contrapositive, assuming that $\preccurlyeq_{\tilde{E}}\neq \preccurlyeq_{k(\tilde{A})}$. Without loss of generality, we can suppose that there are $g, h \in \hom(B_{E}, [0,1])$ such that $g|_{E} \leq h|_{E}$ but $g|_{k(\tilde{A})} \nleq h|_{k(\tilde{A})}$, that is, $g(\tilde{x})\leq h(\tilde{x})$ for all $\tilde{x} \in \tilde{E}$ and $g(k(\tilde{a_{0}})) \nleq h(k(\tilde{a_{0}})) $  for some $\tilde{a_{0}}\in \tilde{A}$. We have
$$T(\varphi_{E}(g))(a_{0})=(\varphi_{A}\circ t \circ \psi_{E}\varphi_{E})(g)(a_{0})=\varphi_{A}(g \circ k)(a_{0})=(g \circ k)(\tilde{a_{0}})$$
and, similarly, $T(\varphi_{E}(g))(a_{0})=(g \circ k)(\tilde{a_{0}})$. So $T(\varphi_{E}(g)) \nleq T(\varphi_{E}(g))$. It follows that $(H_{E}, \tau_{E}, \leq)$ and $(H_{A}, \tau_{A}, \leq)$ are not order isomorphic.
\end{proof}

Thus, to an MV-lattice $A$, there exist an associated MV-space, $(H_{A}, \tau_{A}, \leq)$ and conversely, if $(x,\tau, \leq)$ is a partially ordered MV-space, is possible to associate an MV-lattice, considering the set of increasing clopens of $\tau$, denoted by $\clopI \tau $. These associations define functors in the following way.
 
\subsection*{Functor $\clopI: \PMV \lto \MVL^{\op}$ }

Given an order-preserving continuous map $f: X \to Y$ between partially ordered spaces, the preimage $f \cou: [0,1]^{Y} \lto [0,1]^{X}$  is a homomorphism of positive MV-algebras. 
If, in addition, $x_{1}\leq x_{2} \in X$ and $\alpha \in \clopI\tau_X$, we have  $f(x_{1}) \leq f(x_{2})$ and $f\cou(\alpha)(x_{1})= \alpha(f(x_{1}))\leq \alpha(f(x_{2}))=f \cou(\alpha)(x_{2})$. Then, using again the continuity of $f$, we have that $f\cou|_{\clopI\tau_{Y}}: \clopI\tau_{Y} \lto \clopI\tau_{X}$ is a homomorphism of  MV-lattices.

Now, by defining $\clopI f=f\cou|_{\clopI\tau_{Y}}$ we have that the correspondence
\begin{center}
	$(X, \tau_{X},\leq_{X}) \mapsto \clopI \tau_{X}$, \quad $f \mapsto \clopI f$
\end{center}
defines a  functor $\clopI: \PMV \lto \MVL^{\op}$. Indeed, if we consider
$$\xymatrix{X \ar@<+.4ex>[r]^f & Y \ar@<+.4ex>[r]^g & Z}$$ in $\PMV$ and apply $\clopI$, we obtain that, for each $\alpha \in \clopI \tau_{Z}$,
\begin{align*}
\clopI(g \circ f)(\alpha)=& \alpha \circ g \circ f\\
=&f\cou(g\cou (\alpha))\\
=& \clopI f \circ \clopI g(\alpha)
\end{align*}
and, therefore, $\clopI$ preserves the composition of morphisms. It is clear also that $\clopI\id_{X}=\id_{\clopI\tau_{X}}$.
\begin{remark}
    With an abuse of notation, the image of the space $ (X,\tau_{X}, \leq_{X}) $ under the  functor  $ \clopI $ will be usually denoted, for the sake of readability, as either $\clopI X$ or $\clopI\tau_X$.
\end{remark}

\subsection*{Functor $\Upsilon:\MVL^{\op} \lto \PMV$}

Let $ A $ and $ C $ be MV-lattices, and $(H_{A}, \tau_{A}, \leq_A)$ and $(H_{C}, \tau_{C}, \leq_{C})$ the MV-spaces associated with $ A $ and $ C $ respectively. Then, if  $q:A \lto C$ is a homomorphism of MV-lattices, the function \begin{align*}
	\Upsilon q: H_{C} &\lto H_{A}\\
	              h  & \mapsto h\circ q
\end{align*}
is continuous and order-preserving.  Indeed, let $h\in H_{C}$ and $w=p(\tilde{t_{1}}, \cdots \tilde{t_{m}}) \in \langle \tilde{A} \rangle$. Then $\Upsilon q \cou(w)= w \circ \Upsilon q \in [0,1]^{H_{C}}$ and
 \begin{align*}
 	\Upsilon q\cou (w)(h)=&w(\Upsilon q(h))\\
                           =&w(h \circ q)\\
                           =&p(\tilde{t_{1}}(h \circ q), \cdots, \tilde{t_{m}}(h \circ q))\\
                           =&p(h \circ q (t_{1}), \cdots h\circ q(t_{m}))\\
                           =&p(\widehat{q(t_{1})}(h), \cdots,\widehat{q(t_{m})}(h))\\
                           =& p(\widehat{q(t_{1})}, \cdots \widehat{q(t_{m})})(h).    
 \end{align*}
So $\Upsilon q\cou(w)=p(\widetilde{q(t_{1})}, \cdots \widetilde{q(t_{m})}) \in \langle \tilde{C} \rangle$. That is, $\Upsilon q\cou (w)$ is an element of a basis of $\tau_{C}$, whence $\Upsilon q$ is a continuous map. Moreover if $h_{1} \leq h_{2} \in H_{C}$, $\Upsilon q(h_{1})=h_{1}\circ q \leq h_{2} \circ q= \Upsilon q(h_{2})$. Thus, $\Upsilon q$ is a morphism in the category $\PMV$. Moreover if $r:C \lto D$ is a homomorphism of MV-lattices, we have that $\Upsilon(r \circ q)(h)=(\Upsilon q \circ \Upsilon r)(h)$. So the assignments
\begin{center}
	$A \mapsto (H_{A}, \tau_{A}, \leq_{A})$,  \quad $q\mapsto \Upsilon q$
\end{center}
defines a functor $\Upsilon: \MVL^{\op} \lto \PMV$.

\begin{theorem}\label{adjuncionP}
	The functor $\clopI:\PMV \rightarrow \MVL^{\op} $ is left adjoint to $\Upsilon: \MVL^{\op} \lto \PMV  $ with unit 
	\begin{align*}
		\eta_{X}: X &\lto \Upsilon\clopI \tau& \\ 
		x& \lmapsto f_{x} 
	\end{align*}
where \begin{align*}
	f_{x} :\clopI \tau& \lto [0,1]\\
	\alpha& \lmapsto \alpha(x)
\end{align*}
	for each Priestley MV-space $(X, \tau)$ and with counit $\e_{A} \in \clopI\Upsilon A$ defined by
	\begin{align*}
		\e_{A}: A &\lto  \clopI \Upsilon A \\ 
		a& \lmapsto \widehat{a}
	\end{align*}
	for each MV-lattice  $A$.
\end{theorem}
\begin{proof}
	We shall prove that the triangles
	\begin{multicols}{2}
		$$\xymatrix{
			&		& \clopI X \ar[dd]^{\e_{\clopI X}} \ar[lldd]_{\id_{\clopI X}}\\
			& & \\
			\clopI X & & \clopI\Upsilon\clopI X\ar[ll]^{\clopI \eta_{X}}	\\
		}$$
		$$\xymatrix{
			\Upsilon A\ar[rr]^{\eta_{\Upsilon A}}  \ar[rrdd]_{\id_{\Upsilon A}}
			& & \Upsilon \clopI \Upsilon A\ar[dd]^{\Upsilon \e_{A}}\\
			& &\\
			&		& \Upsilon A  
		}$$
	\end{multicols}	
	
	commute for all $X \in \PMV$ and $ A \in \MVL$, that is $(\clopI\eta_{X})\circ \e_{\clopI X}= \id_{\clopI X}$ and $(\Upsilon\e_{A})\circ \eta_{\Upsilon A}=\id_{\Upsilon A}$.
    
     Given that $\Upsilon \clopI X=(H_{\clopI \tau_{X}}, \tau_{C}, \leq_{C})$, we   will denote $H_{\clopI \tau_{X}}$ by $H_{C}$ for simplicity. So, we have that
     \begin{align*}\clopI \eta_{X}=\eta_{X}\cou|_{\clopI \tau_{C}}:\clopI \tau_{C}&\lto \clopI X\\
       \beta & \lmapsto \beta \circ \eta_{X}.
     \end{align*} and,  for $\alpha \in \clopI \tau_{X}$ and $x \in X$, it holds that 
	\begin{align*}
		\eta_{X}\cou(\tilde{\alpha})(x)=& \ (\tilde{\alpha}\circ \eta_{X})(x)\\
		=& \ \tilde{\alpha}(\eta_{X}(x))\\
		=& \ \eta_{X}(x)(\alpha)\\
		=& \ \alpha(x).
	\end{align*}
	So $((\clopI\eta_{X})\circ \e_{\clopI X})(\alpha)=(\clopI\eta_{X})(\tilde{\alpha})=\alpha$. 
    
	On the other hand, since $\Upsilon A= (H_{A},\tau_{A}, \leq_{A})$ we have, for $f \in H_{A}$, 
	\begin{align*}
	\eta_{\Upsilon A}: H_{A}& \lto  H_{\clopI\tau_{a}} \\
 f& \lmapsto  \eta_{\Upsilon A}(f)
 \end{align*}
 where
 \begin{align*}
 \eta_{\Upsilon A}(f):\clopI \tau_{A}& \lto [0,1]\\
 \tilde{b}& \lmapsto \tilde{b}(f),
	\end{align*}
 and thus
	$((\Upsilon\e_{A})\circ \eta_{\Upsilon A})(f)= \Upsilon\e_{A}(\eta_{\Upsilon A}(f))=\eta_{\Upsilon A}(f)\circ \e_{A} \in H_{A}$ and
	\begin{align*}
		(\eta_{\Upsilon A}(f)\circ \e_{A})(a)=& \ \eta_{H_{A}}(f)(\tilde{a})\\
		=& \ \tilde{a}(f)\\
		=& \ f(a),
	\end{align*}
	for all $a \in A$. Then $(\Upsilon\e_{A} \circ \eta_{\Upsilon A})(f)=\eta_{\Upsilon A}(f) \circ \e_{A}=f$. 
\end{proof}

We will now show how Priestley Duality extends to positive MV-algebras and Priestley MV-spaces.

\begin{lemma}\label{compacto}
	For any MV-lattice $(A, +, \cdot, \vee, \wedge, 0,1)$, the associated MV-space $(H_{A}, \tau, \leq)$ is compact.
\end{lemma}
\begin{proof}
Since $B= \langle \tilde{A} \rangle$, by Proposition \ref{termo}, every element of $B$ is a finite combination, by sum and product, of elements of $\tilde{A} \cup \tilde{A}^{*}$. So, $\tilde{A} \cup \tilde{A}^{*}$ is a large subbase of $\tau$ (see Definition \ref{subbase MV-top}). In order to prove that 
$(H_{A}, \tau_{A}, \leq)$ is compact, we shall now verify that the hypotheses of Alexander Subbase Lemma for MV-topologies (Lemma \ref{alex}) are satisfied.

Let $W \subseteq \tilde{A}\cup \tilde{A}^{*}$ be a cover of $H_{A}$, that is, $\bigvee W= \1$. If $W$ does not have an additive subcover, then for all $n \in \N$ and $a_{1}, \cdots, a_{n} \in W$, $a_{1}\oplus, \cdots \oplus a_{n}<1$. Thus the ideal $(W]$ of $B$ generated by $W$ is proper and, therefore, is contained in some $M \in \Max B$. Moreover, by Lemma \ref{equipotencia}, $M=\bar{f}^{-1}[0]$ for some MV-algebra homomorphism $\bar{f}:B\lto [0,1]$ and $f \in H_{A}$. So $b(f)=\bar{f}(b)=0$ for all $b \in M$ and $\bigvee W(f)=\bigvee_{b\in W}b(f)\leq \bigvee_{b\in M}b(f)=0$. This implies  $\bigvee W \neq \1$, in contradiction with the hypothesis. Therefore, $(H_{A}, \tau)$ is compact.
\end{proof}

\begin{lemma}\label{P2}The  MV-space $(H_{A}, \tau_{A}, \leq)$ associated with any MV-lattice $A$ is totally order disconnected.
	  \end{lemma}
\begin{proof}
	We must prove that for $f \nleq g$ in $H_{A}$ there exists an increasing clopen $\alpha$ such that $\alpha(f)=1$ and $\alpha(g)=0$. Since $f \not\leq g$, there exists $a \in A$ such that $g(a)<f(a)$. Applying Lemma \ref{separacion}, for $x=g(a)$ and $y=f(a)$, we obtain a term $t$ such that $0=t(g(a))$ and $1=t(f(a))$. Furthermore, since $f$ and $g$ preserve $+$ and $\cdot$, and $t(a) \in A$, we have that $0=t(g(a))=g(t(a))=\widehat{t(a)}(g)$ and $1=t(f(a))=f(t(a))=\widehat{t(a)}(f)$, where $\alpha=\widehat{t(a)}$ is an increasing clopen of $\tau$.
\end{proof}
As a consequence of Lemmas \ref{compacto} and \ref{P2}, along with the fact that the MV-algebra generated by $A$ forms a base of clopen sets for $(H_{A}, \tau_{A}, \leq)$, we have the following result.

\begin{proposition}Let $A$ be an MV-lattice and $(H_{A}, \tau, \leq)$ its associated MV-space. Then $(H_{A}, \tau, \leq)$ is a Priestley MV-space. 
\end{proposition} 
\begin{lemma}
If $(X, \tau, \leq)$ is a partially ordered MV-space, then $A=\clopI \tau$ is $H$-complete.
\end{lemma}
\begin{proof}
    Let $B= \langle A \rangle $ be the MV-algebra generated by $A$, and let $E$ be another  positive MV-algebra generating  $B$ and such that $A \subseteq E$ and $\preccurlyeq_{A}= \preccurlyeq_{E}$. We will prove that $A=E$. To do so, consider an element $e \in E$ and  note that, $e=t(a_{1}, \ldots, a_{n})$ for some term $t$ in the language $\{ \oplus, ^{*}\}$ and $a_{i} \in A$.  Moreover, if $f_{x_{i}}\in H_{A}$ and $\bar{f}_{x_{i}} \in \hom(B, [0,1])$ is its corresponding extension, then
    \begin{align*}
    \bar{f}_{x_{i}}(e)=&\bar{f}_{x_{i}}(t(a_{1}, \ldots, a_{n}))\\
    =& t(\bar{f}_{x_{i}}(a), \ldots , \bar{f}_{x_{i}}(a_{n}))\\ 
    =& t(f_{x_{i}}(a_{1}), \ldots , f_{x_{i}}(a_{n}))\\
    =&t(a_{1}(x_{i}),\ldots , a_{n}(x_{i}))\\
    =& t(a_{1}, \ldots, a_{n})(x_{i})\\
    =& e(x_{i}).
    \end{align*}
    Now, if   $x_{1}\leq x_{2} \in X$, by the proof of the previous lemma, $\eta_{X}$ preserves the order, and thus $f_{x_{1}}\leq f_{x_{2}} \in H_{A}$. So,  $(\bar{f}_{x_{1}}, \bar{f}_{x_{2}}) \in \preccurlyeq_{A}= \preccurlyeq_{E}$ and $e(x_{1})=\bar{f}_{x_{1}}(e) \leq \bar{f}_{x_{2}}(e)=e_{x_{2}}$. It follows that $e \in \clopI\tau=A$ and, therefore, $A=E$.
\end{proof}

\begin{lemma}
If $(X, \tau, \leq)$ is a Priestley MV-space, then $\eta_{X}: X \lto \Upsilon\clopI\tau$ is an order isomorphism. 
\end{lemma}
\begin{proof}
Let $A=\clopI X$, we have that $\Upsilon\clopI X=(H_{A}, \tau_{A}, \leq_{A})$. To simplify the notations, we denote $\Upsilon\clopI X$ simply by $H_{A}$. So
    \begin{align*}
		\eta_{X}: X &\lto  H_{A} \\ 
		x& \lmapsto f_{x} 
	\end{align*}
where \begin{align*}
	f_{x} :A & \lto [0,1]\\
	\alpha& \lmapsto \alpha(x)
\end{align*}
is a homomorphism of positive MV-algebras.

If $x_{1} \leq x_{2}$, then $a(x_{1})\leq a(x_{2})$ for all $a \in A$, that is, $f_{x_{1}}(a)\leq f_{x_{2}}(a)$ for all $a \in A$. Therefore $f_{x_{1}}\leq f_{x_{2}}$ and $\eta_{X}$ preserves the order. On the other hand, if $x_{1}\ \nleq x_{2}$, since $(x, \tau, \leq)$ is totally order disconnected, there exists an increasing clopen $\alpha \in A$ such that $\alpha(x_{1})=1$ and $\alpha(x_{2})=0$. Thus $f_{x_{1}}(\alpha)=1>0=f_{x_{2}}(\alpha)$ and $f_{x_{1}}\nleq f_{x_{2}}$. Therefore $x_{1}\leq x_{2} \in X$ if and only if $f_{x_{1}}\leq f_{x_{2}}$.

Now, to prove that $\eta_{X}$ is onto, let $f \in H_{A}$ and assume, by contradiction, that for all $x \in X$ there exists $\alpha \in A$ such that $f(\alpha) \neq \alpha(x)$. Since $f(\alpha)$ and $\alpha(x)$ belong to $[0,1]$, we have that $f(\alpha) < \alpha(x)$ or $\alpha(x)< f(\alpha)$ . If $f(\alpha)< \alpha(x)$, by Lemma \ref{separacion}, there exists a term 
$t$ in the language $\{\oplus, \odot\}$ such that $0=t(f(\alpha))=f(t(\alpha))$ and $1=t(\alpha(x))=t(\alpha)(x)$, where $t(\alpha)=\beta_{x} \in A$. So, for each $x \in X$, there exists $\beta_{x} \in A$ such that $\beta_{x}(x)=1$ and $f(\beta_{x})=0$. Then, $\bigvee\limits_{x \in X}\beta_{x}= \n{1}$, that is, the family $\{\beta_{x}\}_{x \in X}$ is a open covering of $X$ and since $X$ is compact, $\{\beta_{x}\}_{x \in X}$ contains an additive covering $\{\beta_{x_{i}}\}_{i=1}^{n}$. So, $f(\n{1})= f(\beta_{x_{1}}\oplus \cdots \oplus \beta_{x_{n}})=f(\beta_{x_{1}})\oplus \cdots \oplus f(\beta_{x_{n}})=0$ which is impossible, because $f(\n{1})=1$. Similarly, when $\alpha(x)< f(\alpha)$, there is a term $\gamma$ in the language $\{\oplus, \odot\}$ such that $1=\gamma(f(\alpha))= f(\gamma (\alpha))$ and $0=\gamma(\alpha(x))= \gamma(\alpha)(x)$. Taking $\theta_{x}^{*}=\gamma(\alpha) \in A$, we have that $\theta_{x}(x)=1$, $f(\theta_{x}^{*})=1$ and $\{\theta_{x}\}_{x \in X}$ is an open covering which contains an additive covering $\{\theta_{x_{i}}\}_{i=1}^{n}$. So, $\theta_{x_{1}} \oplus \cdots \oplus \theta_{x_{n}}=1$ and, since $\theta_{x_{i}}^{*} \in A$, 
\begin{align*}
    f(\n{0})=&f((\theta_{x_{1}} \oplus \cdots \oplus \theta_{x_{n}})^{*})\\
    =&f(\theta_{x_{1}}^{*}\odot\cdots \odot \theta_{x_{n}}^{*})\\
    =& f(\theta_{x_{1}}^{*}) \odot \cdots \odot f(\theta_{x_{n}}^{*})\\
    =&1
\end{align*}
which is not possible because $f(\n{0})=0$. Therefore $\eta_{X}$ is onto.
\end{proof}
\begin{lemma} Let $(H_{A}, \tau, \leq_{A})$  be the partially ordered MV-space associated with an $\lcc$ positive MV-algebra $A$ and such that $\tilde{A}$ is $H$-complete. 
Then the set of increasing clopens of $\tau_A$, is equal to $\tilde{A}$. 
\end{lemma}
 This lemma indicates that if $A$ is an H-complete lcc MV-algebra, then $\e_{A}$ is an isomorphism of positive MV-algebras. 
\begin{proof}Let $(H_{A},  \tau_{A}, \leq)$ be the MV-space associated to $A$ and let $B\subseteq [0,1]^{H_{A}}$ be the MV-algebra generated by $\tilde{A} $. Since $B$ is $\lcc$, $B=\clop\tau_{A}$ and by Proposition \ref{A}, $A \cong \tilde{A}$. Morevoer, since  $\tilde{A} \subseteq \clopI\tau_{A}$, then
$\clopI\tau_{A}$ is also a positive MV-algebra generating $B$ and $\preccurlyeq_{\clopI\tau_{A}}\subseteq\preccurlyeq_{A}$. To show the other inclusion, we assume $(\bar{f_{1}},\bar{f_{2}}) \in \preccurlyeq_{A}$ and $\alpha \in \clopI\tau_{A}$. Thus $\bar{f_{1}}|_{A} \leq \bar{f_{2}}|_{A}$, and the corresponding $f_{1}, f_{2} \in H_{A}$ are such that $f_{1}\leq f_{2}$. So,  $\bar{f_{1}}(\alpha)=\alpha(f_{1})\leq \alpha(f_{2})=\bar{f_{2}}(\alpha)$. Therefore $(\bar{f_{1}},\bar{f_{2}})\in \preccurlyeq_{\clopI\tau_{A}}$ and, since $ \bigcap\limits_{f \in H_{A} } \ker f = \Delta $ and $\tilde{A}$ is $H$-complete, we conclude that $A\cong\tilde{A}=\clopI\tau_{A}$.
\end{proof}
Let $\MVlcc_{+H}$ be the category of $\lcc$ positive MV-algebras that are $H$-complete and $\Pries^{\lcc}$ be the category of Priestley MV-spaces whose  set of increasing clopen sets is an $\lcc$ positive MV-algebra. Then, thanks to Theorem \ref{adjuncionP} and the last two lemmas, the following result readily follows.
\begin{theorem}\label{dualth}
The restrictions of $\clopI$ and $\Upsilon$ form a duality between $\MVlcc_{+H}$ and $\Pries^{\lcc}$.
\end{theorem}

\section{Comparing the dualities}\label{compsec}

In this section, we show that the duality developed in the previous section extends both classical Priestley Duality and the Stone-type duality for MV-algebras.

On the one hand, our duality is a proper extension of the classical one, in the sense that the categories involved in the latter are full subcategories of the ones involved in our duality, and the restrictions of the functors defined in the previous section to such subcategories yield the classical Priestley Duality.

On the other hand, just like classical Priestley Duality extends Stone Duality, our duality extends the one between lcc MV-algebras and Stone MV-spaces of \cite{ciro2016}.

In order to illustrate the relationship between $(\Upsilon,\clopI)$ and classical Priestley Duality, we first observe that, if $A$ is a bounded distributive lattice thought of as an MV-lattice, we have that $+$ and $\cdot$ coincide with $\vee$ and $\wedge$ respectively. Therefore, if $f: A \to [0,1]$ is an MV-lattice homomorphism, for all $a \in A$, $f(a)$ must be an additively idempotent element of $[0,1]$, and this implies that $f[A] = \{0,1\}$. So, homomorphisms of $A$ with values in $[0,1]$ are in bijective correspondence with those with values in $\{0,1\}$, and it is immediate to see that everything collapses to classical Priestley Duality.

For what concerns the relationship of our duality with the one in \cite{ciro2016}, let $A$ be an MV-lattice and $B \subseteq [0,1]^{H_{A}}$ the MV-algebra generated by $\tilde{A}$. Since $B$ is a semisimple MV-algebra, $B$ is isomorphic to an MV-subalgebra $\bar{B}$ of the MV-algebra of fuzzy subsets $[0,1]^{Max B}$ 
and $\bar{B}$ is a base of an MV-topology $\Omega_{B}$ over $\Max B$. In this way, we obtain the Stone MV-space $(\Max B, \Omega_{B})$, i.e., the image of $B$ under the functor $\Max: \MV \lto \TMV$ defined in \cite{ciro2016}.

It is worth noticing that the fuzzy version of Stone duality presented in \cite{ciro2016} was the starting point of a sheaf representation of a large class of MV-algebra and a sheaf embedding of all MV-algebras, published in \cite{dlprsoco}. The introduction of positive MV-algebra is very recent, therefore, the algebraic theory is not sufficiently developed yet to allow a similar sheaf representation as an easy consequence of our result. However, the good behaviour of the structure suggests that such a result may be possible in the next future.

\begin{proposition}\label{homeoprop}
The MV-spaces $(H_{A}, \tau_{A})$ and $(\Max B, \Omega_{B})$ are homeomorphic.
	\end{proposition}
\begin{proof} For each $f \in H_{A}$, let us consider $\bar{f}\in \hom_{\MV}(B, [0,1])$, the extension of $f$ as in the Lemma \ref{equipotencia}, and denote the maximal ideal $\bar{f}^{-1}[0]$ by $M_{\bar{f}}$.

By Lemma \ref{equipotencia}, the map 
\begin{align*}
    F:H_{A} & \lto \Max B\\
    f & \mapsto \bar{f}^{-1}[0]
\end{align*} 
is bijective. To show that $F$ is an open and continuous map, we consider a basic open set of $\O_{B}$  
\begin{align*}    
\hat{b}: \Max B &\lto [0,1] \\
             M& \mapsto \iota_{M}(b/M)  
\end{align*} where $\iota_{M}:B/M \lto [0,1]$ is the embedding as in Theorem \ref{bell repr}. 
Note that, by definition, $\bar{f}(b)=b(f)$ for all $b \in B \subset [0,1]^{H_{A}}$. Moreover, by Remark \ref{B[0,1]}, $\bar{f}(b)=\iota_{M_{\bar{f}}}(b/M_{\bar{f}})$. So
$$F\cou(\hat{b})(f)=\hat{b}\circ F(f)=\hat{b}(M_{\bar{f}})=\iota(b/M_{\bar{f}})=\bar{f}(b)=b(f).$$
	 That is $ F\cou(\bar{b})=b \in B \subset \tau_{A}$ and $F$ is continuous. On the other hand, for $b \in B \subset \tau_{A}$ and $M \in \Max B$, $F ^{\rightarrow}(b) \in [0,1]^{\Max B} $ is such that $$F ^{\rightarrow}(b)(M)=\bigvee\limits_{F(g)=M}b(g)=b(h)=\bar{h}(b)= \iota_{M}(b/M)=\hat{b}(M)$$ where $M=\bar{h}^{-1}[0]$. Thus $F^{\rightarrow}(b)=\hat{b} \in \Omega_{B}$  and $F$ is an open map.
\end{proof}
	\begin{corollary}
	    The restriction of the functor $\Upsilon$ to $\lcc$ MV-algebras is naturally isomorphic to the functor $\Max$.
	\end{corollary}
    \begin{proof}  
    If $A$ is an MV-algebra and $f \in H_{A}$, then $0=f(0_{A})=f(a \odot a^{*})=f(a) \odot f(a^{*})$ and  $1=f(1_{A})=f(a \oplus a^{*})=f(a)\oplus f(a^{*})$ for $a \in A$. So $f(a^{*})=f(a)^{*}$, hence $f \in \hom_{\MV}(A, [0,1])$ 
 and $H_{A}=\hom_{\MV}(A,[0,1])$.  Moreover, by Proposition \ref{A}, the map 
 \begin{align*}
     T:A &\lto \tilde{A}\\
     a & \mapsto \tilde{a}
 \end{align*}
  is bijective and, since the operations  are defined pointwise, it is clear that $T$ is an isomorphism of MV-algebras. On the other hand
  $$\tilde{a}(f)=f(a)=\iota_{M_{f}}(a/M_{f})=\hat{a}(M_{f})$$
  where $M_{f}=f^{-1}[0]$ and the order in $H_{A}$ is trivial. Indeed, if $f$ and $g$ in $H_{A}$ are such that $f \leq_{A} g$, $f(a) \leq g(a)$ for all $a \in A$,  then $f(a^{*})\leq g(a^{*})$ which implies that $g(a) \leq f(a)$ and consequently $f=g$. Finally, by Proposition \ref{homeoprop}, we have that $\Upsilon A=(H_{A}, \tau_{A}, \Delta) \cong (\Max A, \Omega_{A}, \Delta)$, therefore, the natural isomorphism between $\Upsilon$ and $\Max$ is obtained by simply forgetting the trivial order relation.
  \end{proof}

  On the other hand, if we consider a Stone MV-space $(X, \tau)$ as a trivially ordered MV-space, then $\clopI\tau=\clop \tau$. So, the functor $\clopI$ coincides with $\clop$ on Stone MV-spaces and thus the restriction of the adjunction $\clopI \dashv \Upsilon$ to $\lcc$ MV-algebras and Stone MV-spaces coincide with Stone duality for MV-algebras.
  
Summing up, we have the following diagram of categorical full embeddings and dualities, with obvious meaning of the arrows, $\mathcal{BDL}$ being the category of bounded distributive lattices with their homomorphisms.
	
$$	
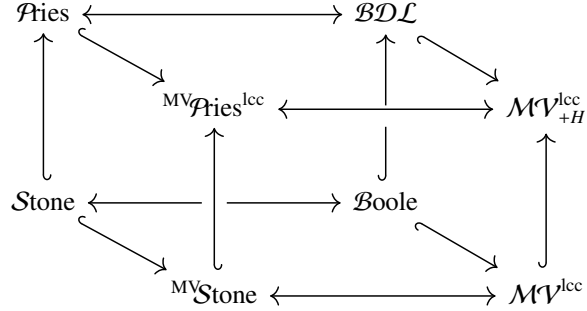
\begin{tikzcd}
\mathcal{P}\!\!\operatorname{ries} \arrow[rd, hook] \arrow[rr, leftrightarrow]      &                                   & \mathcal{BDL} \arrow[rd, hook']                                             &                         \\
                                                                    & \Pries^{\lcc} \arrow[rr, leftrightarrow]          & {} \arrow[u]                                                                & \MVlcc_{+H}             \\
\mathcal{S}\!\operatorname{tone} \arrow[rd, hook] \arrow[uu, hook'] & {} \arrow[l] \arrow[r]            & \mathcal{B}\!\operatorname{oole} \arrow[rd, hook'] \arrow[u, no head, hook] &                         \\
                                                                    & \SMV \arrow[uu, hook'] \arrow[rr, leftrightarrow] &                                                                             & \MVlcc \arrow[uu, hook]
\end{tikzcd}
$$	
  \textbf{Open Question.} The problem of algebraically characterizing $\lcc$ positive MV algebras, as well as describing the Priestley spaces that support this duality, remains open.
\bibliographystyle{acm}

\bibliography{refpriest}
\end{document}